\documentclass[12pt]{amsart}

\usepackage{amsmath,amssymb,amsthm,graphicx}
\usepackage{amscd}
\usepackage{color}
\usepackage{enumitem}

\renewcommand{\d}{\delta}
\newcommand{\g}{\gamma}

\newcommand{\s}{\sigma}

\newcommand{\bd}{\partial} 

\newcommand{\cO}{\mathcal{O}}

\newcommand{\QQ}{\mathbb{Q}}

\newcommand{\CC}{\mathbb{C}}

\newcommand{\PP}{\mathbb{P}}

\newcommand{\ZZ}{\mathbb{Z}}

\newcommand{\x}{\times}

\newcommand{\la}{\langle}
\newcommand{\ra}{\rangle}

\newcommand{\CP}{\CC P}

\newcommand{\orb}{\mathrm{orb}}
\newcommand{\Jac}{\mathrm{Jac}}

\theoremstyle{plain}
\newtheorem{proposition}{Proposition}
\newtheorem{theorem}[proposition]{Theorem}
\newtheorem{lemma}[proposition]{Lemma}
\newtheorem{corollary}[proposition]{Corollary}
\theoremstyle{definition}
\newtheorem{definition}[proposition]{Definition}

\theoremstyle{remark}
\newtheorem{remark}[proposition]{Remark}
\newtheorem{question}[proposition]{Question}
\newtheorem{conjecture}[proposition]{Conjecture}

\title[Quasi-regular Sasakian structures]{Quasi-regular Sasakian and K-contact structures on Smale-Barden manifolds}

\author[A. Ca\~nas]{Alejandro Ca\~nas}
\address{Departamento de \'Algebra, Geometr\'ia y Topolog\'ia, Facultad de Ciencias, Universidad de M\'alaga, Campus de Teatinos s/n, 29071 Málaga, Spain}\email{alejandro.cm.95@uma.es}

\author[V. Mu\~{n}oz]{Vicente Mu\~{n}oz}
\address{Departamento de \'Algebra, Geometr\'ia y Topolog\'ia, Facultad de Ciencias, Universidad de M\'alaga, Campus de Teatinos s/n, 29071 Málaga, Spain}\email{vicente.munoz@uma.es}

\author[M. Sch\"utt]{Matthias Sch\"utt}
\address{Institut f\"ur Algebraische Geometrie, Gottfried Wilhelm Leibniz Universit\"at Hannover, Welfengarten 1, 30167 Hannover, Germany}
\email{schuett@math.uni-hannover.de}

\author[A. Tralle]{Aleksy Tralle}
\address{Faculty of Mathematics and Computer Science, University of Warmia and Mazury, S\l\/oneczna 54, 10-710 Olsztyn, Poland}
\email{tralle@matman.uwm.edu.pl}



\begin{document}
\maketitle

\begin{abstract} 
Smale-Barden manifolds are simply-connected closed $5$-manifolds. It is an important and difficult question 
to decide when a Smale-Barden manifold admits a Sasakian or a K-contact structure.
The known constructions of Sasakian and K-contact structures are obtained mainly by two techniques. These are  either links (Boyer and Galicki), 
or semi-regular  Seifert fibrations over smooth orbifolds (Koll\'ar). Recently, 
the second named author of this article started the systematic development of quasi-regular Seifert fibrations, that is,  
over orbifolds which are not necessarily smooth. The present work is devoted to several applications of this theory. 
First, we develop constructions of a Smale-Barden manifold admitting a quasi-regular Sasakian structure but not a semi-regular 
K-contact structure. Second, we determine all Smale-Barden manifolds that admit a null Sasakian structure. Finally, we
show a counterexample in the realm of cyclic K\"ahler orbifolds to the algebro-geometric conjecture in \cite{MRT} that
claims that for an algebraic surface with $b_1=0$ and $b_2>1$ there cannot be $b_2$ smooth disjoint complex curves
of genus $g>0$ spanning the (rational) homology.
\end{abstract}  

\section{Introduction}

Sasakian and K-contact geometry are topics of great interest for researchers in the fields of 
differential geometry, algebraic geometry and topology. The main object is defined as follows.
Consider a contact co-oriented manifold $(M,\eta)$ with a contact form $\eta$. We say that $(M,\eta)$ admits a {\it Sasakian structure} $(M,g,\xi,\eta,J)$ if:
\begin{itemize}[leftmargin=.3in]
\item there exists an endomorphism $J:TM\rightarrow TM$ such that 
 $J^2=-\operatorname{Id}+\xi\otimes\eta$,
for the Reeb vector field $\xi$ of $\eta$,
\item $J$ satisfies the conditions
$d\eta(JX,JY)=d\eta(X,Y)$,
for all vector fields $X,Y$ and $d\eta(JX,X)>0$ for all non-zero $X\in\ker\eta$,
\item the Reeb vector field $\xi$ is Killing with respect to the Riemannian metric 
$g(X,Y)=d\eta(JX,Y)+\eta(X)\eta(Y),$
\item the almost complex structure $I$ on the contact cone 
$C(M)=(M\times\mathbb{R}_{+},t^2g+dt^2)$
 defined by 
$I(X)=J(X),X\in\ker \eta, I(\xi)=t{\partial\over\partial t},I\left(t{\partial\over\partial t}\right)=-\xi$,
 is integrable.
\end{itemize}

If one drops the condition of the integrability of  $I$, one obtains a {\it K-contact} structure.

The seminal book \cite{BG} shows several important directions of research and still unsolved problems related to manifolds endowed with such structures. One can mention the  problems of existence of K-contact/Sasakian structures and the research program of studying topological properties of manifolds endowed with these. Recent papers \cite{BFMT, CMRV,  CDY, CDMY1, CDMY2, HT, MRT, MT} contribute to the topological program. Existence questions were analyzed in  the foundational papers of Koll\'ar \cite{K, K1} which showed that studying Sasakian manifolds essentially amounts to constructing Seifert bundles over K\"ahler or symplectic orbifolds. This technique proved to be very efficient and resulted, for example, in substantial progress in understanding the Smale-Barden manifolds with Sasakian structures \cite{CMRV, M, MRT, MT}. However, several important problems are still not solved, essentially because the known results \cite{K,K1, MRT, MT} are obtained for a smaller class of semi-regular Sasakian or K-contact structures, that is, determined by Seifert bundles over smooth orbifolds. These bundles are called semi-regular. If one allows for more general singularities one comes to the notion of a quasi-regular Seifert bundle. The first development of this more general theory was begun by the second author of this article in \cite{M}. In the present paper we further develop the construction techniques for quasi-regular Sasakian or K-contact structures and compare them with the semi-regular case. 
 
 A $5$-dimensional simply connected manifold $M$ is called a {\it Smale-Barden manifold}. These manifolds are classified by their second homology group over $\ZZ$ and the so-called {\it Barden invariant} \cite{B, S}. 
 In more detail,
let $M$ be a compact smooth oriented simply connected $5$-manifold. 
Let us write $H_2(M,\ZZ)$ as a direct sum of cyclic group of prime  power order
  \begin{equation}\label{eqn:H2-1}
  H_2(M,\ZZ)=\ZZ^k\oplus \big( \mathop{\oplus}\limits_{p,i}\, \ZZ_{p^i}^{c(p^i)}\big),
  \end{equation}
where $k=b_2(M)$. Choose this decomposition in a way that the second Stiefel-Whitney class map
  $w_2: H_2(M,\ZZ)\rightarrow\ZZ_2$
is zero on all but one summand $\ZZ_{2^j}$. The value of $j$ is unique, it
is denoted by $i(M)$ and is called the Barden invariant. The fundamental question arises, which Smale-Barden manifolds admit K-contact or Sasakian structures. In \cite{MT} the second and the fourth author of this article obtained classification results for Smale-Barden manifolds with semi-regular Sasakian structures. Thus, one of the aims of this work is to analyze the topic ``quasi-regular Sasakian manifolds versus semi-regular". In particular, we ask the following question.

\begin{question} \label{Q1}
Are there Smale-Barden manifolds which carry quasi-regular Sasakian structures but do not carry semi-regular Sasakian structures?
\end{question}

We answer this question in positive in Theorem \ref{thm:quasi-vs-semi}. Moreover, the examples that we provide also satisfy that they
do not admit
semi-regular K-contact structures. 
In general, we have the following chain of inclusions:
$$
 \begin{array}{ccc}
   \{\text{Sasakian semi-regular} \} & \subset  &  \{\text{Sasakian quasi-regular} \}  \\
    \cap & & \cap \\
   \{\text{K-contact semi-regular} \} & \subset  &  \{\text{K-contact quasi-regular} \}  
   \end{array}
   $$   
 In \cite{CMRV} we have constructed the first Smale-Barden manifold which is K-contact semi-regular 
 but not Sasakian semi-regular.  By Corollary \ref{cor:extra}, we have a Sasakian quasi-regular
 which is not Sasakian semi-regular. The rightmost inclusion is open (see Question \ref{Q2} below).
 In Theorem \ref{thm:quasi-vs-semi} of this article we show  an example of a manifold with a Sasakian quasi-regular structure  which does not admit  K-contact semi-regular structures,
 so all other inclusions are strict.
   
The following problem was posed by Boyer and Galicki \cite{BG}: 
\begin{question} \label{Q2}
Are there Smale-Barden manifolds with K-contact but no Sasakian structures?
\end{question} 
The answer to this question is still not known, although substantial progress has been achieved in \cite{MRT} and the answer in the class of semi-regular Seifert bundles was found in \cite{CMRV}, where it is constructed
a Smale-Barden manifold with a semi-regular K-contact structure that does not admit a semi-regular Sasakian structure.
The key to the construction of such a semi-regular K-contact manifold $M$
is to find a simply-connected smooth symplectic $4$-manifold $X$ with a collection of embedded disjoint
symplectic surfaces $D_i$ of positive genus $g_i>0$ which span the homology $H_2(X,\QQ)$. The manifold $M$ is the Seifert bundle over $X$. 
Such $M$ does not admit a semi-regular Sasakian structure because such configuration of complex curves is not admissible in a complex surface.
Actually, in \cite{MRT} we find the following conjecture.

\begin{conjecture} \label{conj}
There does not exist a smooth complex surface $X$ with $b_1=0$ and $b_2>1$ such that there are smooth disjoint complex
curves $D_i$ of positive genus $g_i>0$ which span the homology $H_2(X,\QQ)$.
\end{conjecture}

Some cases where the conjecture holds appear in \cite{CMRV,MRT}, which are enough for giving the examples of manifolds not admitting
semi-regular Sasakian structures. To find a full answer to Question \ref{Q2} (that is, in the quasi-regular case) one needs to develop the techniques of constructing symplectic and K\"ahler non-smooth orbifolds with second homology spanned by symplectic surfaces or complex curves of positive genus. 
Examples of such orbifolds and the corresponding Seifert bundles are constructed in Section \ref{sec:examples}.   In particular,
we show that Conjecture \ref{conj} does not hold if we assume $X$ to be a complex cyclic orbifold (Theorem \ref{thm:g=1}).

Our last objective is to settle the problem of existence of null Sasakian structures on Smale-Barden manifolds.     
Recall that the Reeb vector field $\xi$ on a co-oriented contact manifold $(M,\eta)$ determines a $1$-dimensional foliation 
$\mathcal{F}_{\xi}$ called the {\it characteristic foliation}. If we are given a Sasakian manifold $(M,g,\xi,\eta,J)$, then 
one can define  the {\it basic Chern classes} $c_k(\mathcal{F}_{\xi})$ of $\mathcal{F}_{\xi}$ which are elements of the 
basic cohomology $H^{2k}_B(\mathcal{F}_{\xi})$ (see \cite[Theorem/Definition 7.5.17]{BG}).     
We say that a Sasakian structure is positive (negative) if $c_1(\mathcal{F}_{\xi})$ can be represented by a positive 
(negative) definite $1$-form. A Sasakian structure is called null, if $c_1(\mathcal{F}_{\xi})=0$. If none of these
cases persists, it is called indefinite. 

In \cite{BG} it is shown that if a Smale-Barden manifold $M$ admits a null Sasakian structure, then 
$M$ is homeomorphic to the connected sum of at most $21$ copies of $S^2\times S^3$. Moreover, 
the authors prove that any $M=\#_k (S^2\times S^3)$ with $2\leq k\leq 21$ admits a null Sasakian structure, 
except, possibly, $b_2(M)=2$ and $b_2(M)=17$.  The following question is left open in \cite{BG}.

\begin{question}[{\cite[Open Problem 10.3.2]{BG}}] \label{Q3}
Find examples of null Sasakian structures on $\#_2(S^2\times S^3)$ and $\#_{17}(S^2\times S^3)$ or show that none can exist.
\end{question}
Note that this question is motivated by applications to $\eta$-Einstein Sasakian geometry \cite{BGM}. Since the transverse geometry of the characteristic foliation $\mathcal{F}_{\xi}$ is K\"ahler, one can understand Sasakian geometry in analogy with K\"ahler geometry. In particular, a null  Sasakian structure can be understood as a transversal Calabi-Yau structure. Therefore, manifolds with null Sasakian structures are odd-dimensional analogues of the Calabi-Yau spaces. Also, every null Sasakian structure  can be deformed to an $\eta$-Einstein structure (see \cite{BGM} for a discussion of possible applications in physics).  

The case $\#_{17}(S^2\times S^3)$ was settled in \cite{CV} using a generalization of the initial approach of Boyer and Galicki via Seifert bundles over weighted surface complete intersections in weighted projective spaces  \cite{IF, R}. Our approach is different and  enables us to provide a complete solution of the problem. This solution illustrates our general  techniques. We answer Question \ref{Q3} in the positive:

\begin{theorem}\label{thm:main-null} 
Any $M=\#_k(S^2\times S^3)$, $2\leq k\leq 21$  admits a null Sasakian structure.
\end{theorem}

Note that the formulation of Theorem A in the introduction of \cite{CV} is incorrect, the correct formulation should 
be ``$\#_k(S^2\times S^3)$ with $3\leq k\leq 21$ admits a null Sasakian structure".

Our basic reference containing elliptic surfaces, complex surfaces, and desingularization process is \cite{GS}.

\subsection*{Acknowledgements} 
The first author is supported by a PhD grant from Universidad de M\'alaga.
The second author was partially supported by Project MINECO (Spain) PGC2018-095448-B-I00. 
The fourth author was supported by the National Science Center (Poland), grant NCN no.\ 2018/31/D/ST1/00053. 
The fourth author started working on the topic during his stay at the Institut de Hautes \'Etudes Scientifiques at Bur-sur-Yvette. He is grateful to the Institute for the  support and the stimulating research atmosphere.

\section{Preliminaries on Seifert bundles}

We freely use the notion of cyclic orbifolds \cite{BG, K,K1, M}. Our basic reference will be \cite{M}. 

\begin{definition}[\cite{M}] \label{def:sing-mnfld} A cyclic singular symplectic (K\"ahler)  manifold is a symplectic cyclic 4-orbifold whose isotropy set is of dimension zero (that is, a finite set $P$ of points, called the singular set).
\end{definition}

For a cyclic singular symplectic 4-manifold a singular point is an isolated isotropy point $x\in P\subset X$. A local model around $x$ is of the form $\mathbb{C}^2/\mathbb{Z}_d$, where $\xi=\exp({2\pi i/d})$ acts as $\xi\cdot (z_1,z_2)=(\xi^{e_1}z_1,\xi^{e_2}z_2)$, where $\operatorname{gcd}(e_1,d)=\operatorname{gcd}(e_2,d)=1$. We will write $d=d(x)$.

\begin{definition}[\cite{M}] \label{def:sing-curve}
A sing-symplectic surface is a symplectic 2-orbifold $D\subset X$ such that if $x\in D$ is a singular point, then $D$ is fixed by the isotropy subgroup. 
Two sing-symplectic surfaces $D_1,D_2$ intersect nicely if at every intersection point $x\in D_1\cap D_2$ there is an adapted Darboux chart $(z_1,z_2)$ centered at $x$ such that $D_1=\{(z_1,0)\}$ and $D_2=\{(0,z_2)\}$ in a model $\mathbb{C}^2/\mathbb{Z}_d$, 
where $\mathbb{Z}_d<U(1)\x U(1)$.
\end{definition}

One can construct cyclic symplectic or K\"ahler orbifolds using the result below.

\begin{proposition}[\cite{M}]\label{prop:orb-constr} 
Let $X$ be a cyclic singular 4-manifold with the set of singular points $P$. Let $D_i$ be embedded sing-symplectic surfaces intersecting nicely, and take coefficients $m_i>1$ such that $\operatorname{gcd}(m_i,m_j)=1$, if $D_i$ and $D_j$ have a non-empty intersection. Then there exists an orbifold $X$ with isotropy surfaces $D_i$ of multiplicities $m_i$, and singular points $x\in P$ of multiplicity $m=d(x)\prod_{i\in I_x}m_i$, $I_x=\{i\,|\,x\in D_i\}.$
\end{proposition}

Our basic tool is a Seifert bundle over a sing-symplectic orbifold. 

\begin{definition} Let $X$ be a cyclic, oriented $n$-orbifold. A Seifert bundle over $X$ is an oriented $(n+1)$-dimensional manifold $M$ equipped with a smooth $S^1$-action and a continuous map $\pi:M\rightarrow X$ such that for an orbifold chart $(U,\tilde U,\mathbb{Z}_m,\varphi)$, there is a commutative diagram
$$
\CD
(S^1\times\tilde U)/\mathbb{Z}_m @>{\cong}>>\pi^{-1}(U)\\
@VVV @V{\pi}VV\\
\tilde U/\mathbb{Z}_m @>{\cong}>> U,
\endCD
$$
where the action of $\mathbb{Z}_m$ on $S^1$ is by multiplication by $\xi=\exp ({2\pi i/ m})$ and the top diffeomorphism is $S^1$-equivariant.
\end{definition}

The basic invariant of the Seifert bundle is the orbifold first Chern class.

\begin{definition} 
For a Seifert bundle $\pi: M\rightarrow X$, we define its first Chern class as follows. Let $l=\operatorname{lcm}(m(x)\,|\,x\in X)$. Denote by $M/l$ the quotient of $M$ by $\mathbb{Z}_l\subset S^1$. Then $M/l\rightarrow X$ is a circle fiber  bundle with the first Chern class $c_1(M/l)\in H^2(X,\mathbb{Z})$. Define
 $$
  c_1(M)={1\over l}c_1(M/l)\in H^2(X,\mathbb{Q}).
  $$
\end{definition}
\begin{definition} We say that an element $a$ in a free abelian group $A$ is primitive, if it cannot be represented as $a=kb$ with a non-trivial $b\in A, k\in\mathbb{N}$.
\end{definition}
Here is the main construction  tool of this paper.

\begin{theorem}[{\cite[Lemma 39]{M}}]\label{thm:constr-tool} 
Let $(X,\omega)$ be a cyclic symplectic 4-orbifold with a collection of embedded symplectic surfaces $D_i$ intersecting nicely, and integer numbers $m_i>1$, with $\operatorname{gcd}(m_i,m_j)=1$ whenever $D_i\cap D_j\not= \emptyset$. Assume that there are local invariants $(m_i,j_i,j_x)$, $x\in P$. Let $b_i$ be integers such that $j_ib_i\equiv 1\,(\operatorname{mod}\,m_i), m=\operatorname{lcm}(m_i)$. Then there is a Seifert bundle $\pi: M\rightarrow X$ such that
\begin{enumerate}
\item It has the first orbifold Chern class $c_1(M)=[\hat\omega]$ for an orbifold symplectic form $\hat\omega$.
\item If $\sum_i{b_im\over m_i}[D_i]\in H^2(X-P,\mathbb{Z})$ is primitive and the second Betti number $b_2(X)\geq 3$, then we can further have that $c_1(M/m)\in H^2(X-P,\mathbb{Z})$ is primitive.
\end{enumerate}
\end{theorem}

This combines with the folllowing basic result for characterizing K-contact and Sasakian structures.

\begin{theorem}[{\cite[Theorems 7.5.1, 7.5.2]{BG}}]\label{thm:seifert.} 
Let  $(M,g,\xi,\eta,J)$ be a quasi-regular Sasakian manifold. Then the space of leaves $X$ of the foliation 
determined by the Reeb field $\xi$ has a natural structure of a cyclic Kähler orbifold, and the projection $M\rightarrow X$ is a Seifert bundle.
Conversely, if $(X,\omega)$ is a Kähler cyclic orbifold and $M$ is the total space of the Seifert bundle determined by the class $[\omega]$, 
then $M$ admits a quasi-regular Sasakian structure.
\end{theorem}

In the same way, one can characterize quasi-regular K-contact manifolds considering symplectic orbifolds 
instead of K\"ahler ones \cite[Theorems 19 and 21]{MRT}.

Note that by \cite{Ruk}, any manifold which admits a Sasakian or K-contact structure, admits a quasi-regular one. Therefore, since we are interested in the existence questions, we can and we will assume that we are dealing with quasi-regular structures.

\medskip

Finally, let us formulate a result which guarantees that the Seifert bundle $M\rightarrow X$ constructed in Theorem \ref{thm:constr-tool} has $H_1(M,\ZZ)=0$. 

\begin{theorem}[{\cite[Theorem 36]{M}}]\label{thm:H1=0} 
Suppose that $\pi:M\rightarrow X$ is a quasi-regular Seifert bundle over a cyclic orbifold $X$ with isotropy surfaces $D_i$ and set of singular points $P$.  Let $m=\operatorname{lcm}(m_i)$. Then $H_1(M,\mathbb{Z})=0$ if and only if
\begin{enumerate}
\item $H_1(X,\mathbb{Z})=0$,
\item $H^2(X,\mathbb{Z})\rightarrow \oplus H^2(D_i,\mathbb{Z}_{m_i})$ is surjective,
\item $c_1(M/m)\in H^2(X-P,\mathbb{Z})$ is a primitive class.
\end{enumerate}
Moreover, $H_2(M,\mathbb{Z})=\mathbb{Z}^k\oplus\big(\mathop{\oplus}\limits_i\mathbb{Z}_{m_i}^{2g_i}\big)$, 
$g_i=$\,genus of $D_i$, $k+1=b_2(X)$.
\end{theorem}

For the convenience of references we will also interchangeably  follow notation and terminology of \cite{BG} and \cite{K, K1}. If $X$ is a cyclic orbifold with singular set $P$ and a family of surfaces $D_i$ we will say that we are given a divisor $\cup_iD_i$, with multiplicities $m_i>1$. The formal sum
$\Delta=\sum_i(1-{1\over m_i})D_i$  will be called the {\it branch divisor}. By definition, the orbifold fundamental group $\pi_1^{\orb}(X)$ is defined as 
 $$
 \pi_1^{\orb}(X)=\pi_1(X-(\Delta\cup P))/\langle \gamma_i^{m_i}=1\rangle,
 $$
where $\langle\gamma_i^{m_i}=1\rangle$ denotes the following relation on $\pi_1(X-(P\cup\Delta))$: for any small loop $\gamma_i$ around a surface  $D_i$ in the branch divisor, one has $\gamma_i^{m_i}=1$.
We will systematically use without notice the following exact sequence (see the general formulation in \cite[Theorem 4.3.18]{BG}) .
If $M\rightarrow X$ is a Seifert fibration, then there is an exact sequence
  $$
  \ldots\rightarrow \pi_1(S^1)=\mathbb{Z}\rightarrow \pi_1(M)\rightarrow \pi_1^{\orb}(X)\rightarrow 1.
  $$
It follows that if $H_1(M,\ZZ)=0$ and $\pi_1^{\orb}(X)=1$, then $M$ must be simply connected. We will use this observation without further notice.

\section{Examples of K\"ahler cyclic orbifolds}\label{sec:examples}

The following appears in \cite[Section III.5]{BPV}.

\begin{lemma} \label{lem:Hirz-Jung}
Consider the action of the cyclic group $\ZZ_m$ on $\CC^2$ given by $(z_1,z_2)\mapsto (\xi z_1,\xi^r z_2)$, where
$\xi=e^{2\pi i /m}$, $0<r<m$ and $\gcd(r,m)=1$.
 Then write a continuous fraction
 $$
  \frac{m}{r}=[b_1,\ldots, b_l]=b_1- \frac{1}{b_2-\frac{1}{b_3- \ldots}}
  $$
The resolution of $\CC^2/\ZZ_m$ has an exceptional divisor formed by a chain of 
smooth rational curves of self-intersections
$-b_1,-b_2,\ldots,-b_l$.
\end{lemma}

Conversely, let $X$ be a smooth surface containing a chain of smooth rational curves $E_1,\ldots, E_l$ of self-intersections
$-b_1,-b_2,\ldots,-b_l$, with all $b_i\geq 2$, intersecting transversally (so that $E_i\cap E_{i+1}$ are nodes, $i=1,\ldots, l-1$).
Let $\pi:X\to \bar X$ be the contraction of $E=E_1\cup\ldots \cup E_l$. Then $\bar X$ has a cyclic singularity at $p=\pi(E)$,
with an action given by Lemma \ref{lem:Hirz-Jung}. Moreover, if $D$ is a curve
intersecting tranversally a tail of the chain (that is, either $E_1$ or $E_l$ at a non-nodal point),
then the push down curve $\bar D=\pi(D)$ is a sing-complex curve (as in Definition \ref{def:sing-curve}). This can be seen 
as follows. 
By \cite[Corollary 2.5]{R1}, if $z_1,z_2$ are the coordinates of $\CC^2$, then the quotient $Z=\CC^2/\ZZ_m$
is described by coordinates $u_0,u_1,\ldots, u_{l+1}$, where 
 $$
  u_0=z_1^m, \qquad u_{i-1}u_{i+1}=u_i^{a_i}\,  \text{ for } i=1,\ldots, l, \qquad u_{l+1}=z_2^m,
 $$
where $\frac{m}{r-m}=[a_1,\ldots, a_l]$. The resolution is obtained by a (weighted) 
blow-up at the origin $(u_0,u_1,\ldots, u_{l+1})=(0,\ldots, 0)$
of  $Z\subset \CC^{l+2}$.
The exceptional divisor is the subset $E$ of a suitable weighted projective space $\PP^{l+1}_w$, 
with coordinates $[w_0,w_	1,\ldots, w_{l+1}]$ parametrized by the equations 
$w_{i-1}w_{i+1}=w_i^{a_i}$, $i=1,\ldots, l$. The divisor $E_i$ is given by the equations $w_0=\ldots=w_{i-2}=0$,
$w_{i+2}=\ldots=w_{l+1}=0$. The strict transform of $\{z_1=0\}$ only intersects $E_1$, and the strict transform
of $\{z_2=0\}$ only intersects $E_l$.

\medskip
 
Our aim now is to construct a K\"ahler cyclic orbifold with the second homology generated by elliptic curves. 
This will provide a counterexample to Conjecture \ref{conj} in the case where $X$ is assumed
to be a complex cyclic orbifold.

\begin{theorem}\label{thm:g=1}
For any $b\geq1$, there is a K\"ahler cyclic orbifold $X$ with $b_1=0$, $b_2=b$ and 
$b$ disjoint curves, all of genus $1$, whose classes span $H^2(X,\QQ)$.
\end{theorem}

\begin{proof}
Consider a regular cubic $C$ inside the projective plane $\CC P^2$. In particular, it has genus $1$. 
Consider now the following short exact sequence of sheaves,
  $$
  0\rightarrow\mathcal{O}_{\CC P^2}\rightarrow\mathcal{O}_{\CC P^2}(C)\rightarrow\mathcal{O}_C(C)\rightarrow0.
  $$
It is known that $h^0(\mathcal{O}_{\CC P^2})=1$, $h^1(\mathcal{O}_{\CC P^2})=0$ and 
$h^0(\mathcal{O}_{\CC P^2}(C))=h^0(\mathcal{O}_{\CC P^2}(3))=10$. Therefore, we get the following exact short sequence in cohomology,
 $$
 0\rightarrow H^0(\mathcal{O}_{\CC P^2})\rightarrow H^0(\mathcal{O}_{\CC P^2}(C))\rightarrow H^0(\mathcal{O}_C(C))\rightarrow 0,
 $$
from where we deduce that $h^0(\mathcal{O}_C(C))=9$. This implies that the linear system 
$|\mathcal{O}_C(C)|$ has dimension $8$. Hence, for any $8$ points chosen in $C$,
we can find a cubic $C'\in|C|$ which intersects $C$ at those points.

Take $p\in C$ and $C'$ a cubic such that $C\cdot C'=8p + q$ for some $q\in C$. 
This defines a map $\varphi:C\to C$ given by $q=\varphi(p)$. We look for a fixed point of $\varphi$.
We describe the map in an alternative way. Fix some base-point $p_0\in C$. This produces an
isomorphism
 $$
 F: C\to \Jac\, C, \quad F(p)= p-p_0
 $$
where $\Jac\, C$ is the Jacobian of degree $0$ divisors. Consider the divisor $D_0=8p_0+\varphi(p_0) \in
 |\mathcal{O}_C(C)|$. All divisors $8p+\varphi(p)$ are equivalent, hence 
 $\varphi(p)\equiv D_0 - 8p$, so 
  \begin{align*}
  F(\varphi(p)) &= \varphi(p)-p_0 = D_0 -8p -p_0= \\ &=  \varphi(p_0)-p_0 -  8(p-p_0)=F(\varphi(p_0))-8 F(p).
  \end{align*}
Hence
  $$
   \varphi(p)= F^{-1}(F(\varphi(p_0))-8F(p)).
  $$
As multiplication (by $k=-8$) and addition in $\Jac\, C$ are morphisms, then $\varphi$ is a morphism. 

To find a fixed point $p\in C$, we note that $\varphi(p)=p$ is equivalent to $F(p)=F(\varphi(p))=F(\varphi(p_0))-8F(p)$,
i.e.\ $9F(p)=F(\varphi(p_0))$. The map 
 $$
 m_9: \Jac\, C\to \Jac\, C, \quad m_9(D)=9D
 $$
given by multiplication by $9$ in the divisors of degree $0$, is a map of degree $81=9^2$. The divisor $s_0=F(\varphi(p_0))$
has $81$ preimages $r_i\in \Jac\,C$, $1\leq i\leq 81$. Then $p_i=F^{-1}(r_i)$ are the $81$ solutions to the
equation $9F(p_i)= 9 r_i= s_0=F(\varphi(p_0))$, that is, to the equation $\varphi(p_i)=p_i$.

Now fix one of these $81$ points $p=p_i$. Then take a
cubic $C'$ such that $C\cdot C'=9p$. Take sections $s,s'\in H^0(\mathcal{O}_{\CC P^2}(3))$ 
with $C$ and $C'$ as zero loci. We then get a pencil $|\langle s,s'\rangle| \subset |\mathcal{O}_{\CC P^2}(3)|$ of cubics intersecting 
at $p$ with multiplicity $9$. Choose $b$ regular cubics inside this pencil, $C_1,\ldots, C_b$. 
These all have genus $1$, self-intersection $C_i^2=9$ 
and any two of them intersect only at $p$, with multiplicity $9$.

Blow-up the plane $\CP^2$, $9$ times at $p$ and denote the resulting surface
by $\widetilde{\CP}{}^2$. That is, we blow-up at $p$, then at the intersection
point of the strict transforms of the $C_i$, and so on. We obtain a chain of $9$ rational curves $E_1,\ldots, E_9$. 
All the exceptional divisors $E_1,\ldots, E_8$ have self-intersection $E_j^2=-2$, $j=1,\ldots, 8$. The
last exceptional divisor $E_9$ satisfies $E_9^2=-1$, and the strict 
transforms of the curves $C_1,\ldots, C_b$, that we denote $\tilde{C}_1,\ldots,\tilde{C}_b$, 
are now pairwise disjoint, each of them 
intersects $E_9$ at one point $x_i$, and all are disjoint from $E_1,\ldots, E_8$. Note that $\tilde C_i^2=0$, $1\leq i\leq b$.
Blow-up $\widetilde{\CP}{}^2$ at the points $x_2,\ldots, x_b$ and denote by $\tilde X$ this new 
surface. Since we obtained $\tilde X$ from $\CC P^2$ by $9+(b-1)$ blow-ups, $b_1(\tilde X)=0$ and $b_2(\tilde X)=9+b$. 
Denote by $\tilde{E}_9$ the strict transform of $E_9$ 
and by $D_i$ the strict transform of $\tilde{C}_i$, $2\leq i\leq b$. Now
$\tilde{E}_9^2=-b$, $D_i^2=-1$, $2\leq i\leq b$, $\tilde{C}_1^2=0$.
It is clear that $\tilde{C}_1,D_2,\ldots, D_b, E_1,\hdots, E_8, \tilde E_9$ form a basis of $H^2(X,\QQ)$.

\begin{figure}[ht]
\begin{center}
\includegraphics[width=8cm]{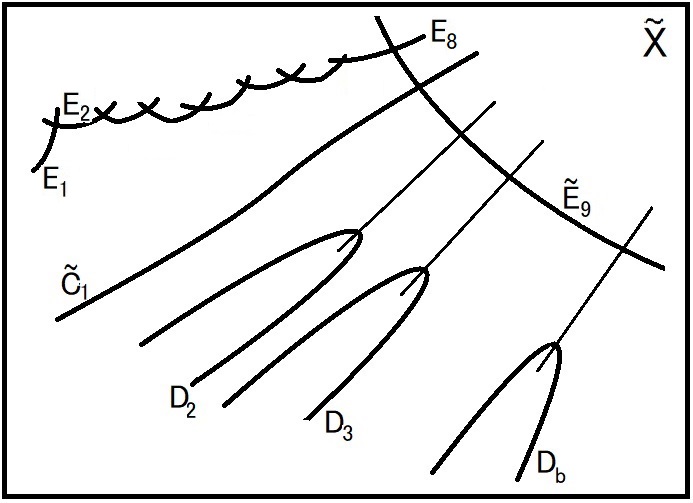}
\end{center}
\end{figure}

Now we contract the chain of $(-2)$-curves given by $E_1, \hdots, E_8,$
extended by the $(-b)$-curve $\tilde E_9$.
By Lemma \ref{lem:Hirz-Jung}, this produces a cyclic singularity. As $[b,2,\stackrel{(8)}\ldots, 2]=\frac{9b-8}{9}$, the
singularity is modelled on $\CC^2/\ZZ_{9b-8}$ with the action of $\xi=\exp(2\pi i /(9b-8))$ via $(z_1,z_2)\mapsto (\xi z_1,\xi^{9}z_2)$.
The result is a cyclic orbifold $X$ with $b_1(X)=0$, $b_2(X)=b$ and $b$ disjoint curves $D_1,D_2,\dots D_b$ all of genus $1$.
Here $D_1$ is the push down of $\tilde{C}_1$ and has self-intersection 
 $D_1^2=\tilde{C}_1^2 +\frac{1}{9b-8}=\frac{1}{9b-8} >0$.
\end{proof}

\begin{remark}
We can simplify the construction in the proof of Theorem \ref{thm:g=1}.
Take a cubic $C$ and a point $p\in C$ which is an inflection. Then let
$L$ be the tangent line, and $C'=3L$ is a cubic and satisfies that $C\cap 3L=9p$. Now take the 
pencil generated by them $|\la C,3L\ra|$. The generic curve is smooth because there is a curve,
namely $C$, that is smooth. All of them intersect at $p$ with multiplicity $9$.

The construction in Theorem \ref{thm:g=1} produces $81$ points. As the number of inflections is $9$,
there are pencils as in Theorem \ref{thm:g=1} which do not come from inflections.
\end{remark}

\begin{remark}
\label{rem:RES}
The surface $\widetilde{\CP}{}^2$ above is a so-called extremal rational elliptic surface,
i.e.\ with finite group of sections.
Similar constructions can be carried out for any other extremal rational elliptic surface
(as classified in \cite{MP}).
Then one can always contract suitable configurations of eight $(-2)$-curves inside the singular fibres
(like $E_1,\hdots,E_8$).
The role of $E_9$ is played by some section.
\end{remark}

\begin{remark}
On Enriques surfaces, the same arrangement of the singular fibres can be achieved
as on the rational elliptic surfaces in Remark \ref{rem:RES}.
The section $E_9$, however, has to be replaced by a rational curve
with $b-1$ nodes, serving as a bisection of the elliptic fibration,
which becomes smooth exactly after blowing up
the intersection points with $b-1$ regular fibres $C_2,\hdots,C_b$.
Such surfaces can be constructed systematically using the base change approach
from \cite{HS}.
Unlike $X$ below in Proposition \ref{prop:pi1orb}, they will not have trivial $\pi_1^{\orb}$.
\end{remark}

We put an orbifold structure on the surface $X$ constructed in Theorem \ref{thm:g=1} by assigning coefficients $m_i$ 
to each $D_i$, $i=1,\ldots, b$, using Proposition \ref{prop:orb-constr}. Let us compute the orbifold fundamental group.

\begin{proposition} \label{prop:pi1orb}
We have $\pi_1^{\orb}(X)=1$.
\end{proposition} 

\begin{proof}
 Let $\pi:\tilde X\to \CP^2$ be the result of the successive blow-ups of $\CP^2$, with the tori $\tilde{C}_1,D_2,\ldots, D_b$ 
 and exceptional curves ${E}_1,\ldots, E_8, \tilde E_9$ and $E_{10},\ldots, E_{8+b}$ coming from the blow-ups at $x_2,\ldots,x_b$.
 We denote $E ={E}_1\cup \ldots\cup E_8\cup \tilde E_9 \cup E_{10}\cup \ldots \cup E_{8+b}$.
 The curves $C_1=\pi(\tilde C_1)$, $C_i=\pi(D_i)$, $i=2,\ldots, b$, 
 are smooth cubic curves intersecting at the point $p=\pi(E)$. Denote $D'=\tilde C_1\cup D_2\cup \ldots \cup D_b \cup E\subset \tilde X$
 and $C=\cup C_i\subset \CP^2$. Then there is an isomorphism 
  $$
  \tilde X-U(D') \cong \CP^2 - U(C),
  $$
 where $U(D'),U(C)$ denote small tubular neighbourhoods of $D',C$, respectively.  
 The fundamental group of $C_i$ is generated by two loops $\alpha_i,\beta_i$. 
 The loops $\alpha_i,\beta_i$ are homotopic to $\alpha_j,\beta_j$ in $U(C)$ (it is enough to construct $C_i$ close
 together and generic in the pencil). Moreover, they are contractible in $U(C)$. 
 For this it is enough to take a curve $C_1$ close to a cuspidal curve, the point $p$
 well away from the cusp, and all other curves $C_i$ close to $C_1$. 

 Now consider the blow-down map 
  $$
\varpi:\tilde X\to X
   $$
and the images $D_1=\varpi(\tilde C_1)$, $D_i=\varpi(D_i)$, $i=2,\ldots, b$. Recall that $D_1$ contains the
singular point $q=\varpi(F)$, where $F= E_1\cup \ldots\cup E_8 \cup \tilde E_9$. 
Let $2\leq i\leq b$. The boundary $\bd U(D_i)$ is 
generated by $\alpha_i,\beta_i$ and a loop $\gamma_i$ around $D_i$.  As $D_i^2=-1$, we have
$\gamma_i^{-1}=[\alpha_i,\beta_i]$. As $\alpha_i,\beta_i$ is contractible in $\tilde X-U(D)$, hence in 
$X- U(\cup D_i)=\tilde X-U(D)$, where $D=\tilde C_1\cup D_2\cup \ldots \cup D_b \cup F\subset D'$,
since $F\subset E$.
There is a surjective map 
 $$
  \pi_1(\tilde X-U(D'))=\pi_1(\CP^2-U(C)) \longrightarrow \pi_1(\tilde X-U(D))=\pi_1(X-U(\cup D_i)).
 $$

Therefore the group $\pi_1(X-U(\cup D_i))$ is generated by a loop around $D_1$, and by a loop going around
a neighbourhood of $q$, which is the image of a loop in $\bd U(F)$. The point $q$ is modeled on a space 
 $\CC^2/\ZZ_d$ and the link is a lens space. Let $x\in \pi_1(\CC^2/\ZZ_d-\{q\})\cong \ZZ_d$ be 
 a generator, where $d=9b-8>1$ is the order of the  singular point.  
 The orbifold fundamental group of $D_1$ is 
  $$
  \pi_1^{\orb}(D_1)=\la \alpha_1,\beta_1,x\,  | \, [\alpha_1,\beta_1] x=1, x^d=1\ra.
  $$
There is a Seifert fibration $S^1\to \bd U(D_1)\to D_1$, which gives an exact sequence
 $$
 0\to  \ZZ\la \gamma_1\ra \to \pi_1(\bd U(D_1)) \to \pi_1^{\orb}(D_1) \to 0.
 $$
 As $D_1^2=\frac1d$, we have that
   $$
   \pi_1(\bd U(D_1))= \la  \alpha_1,\beta_1,x ,\gamma_1\,| \, [\alpha_1,\beta_1] x=1, x^d=\gamma_1,\gamma_1 \text{ central} \ra.
  $$
As we mentioned before, $\alpha_1,\beta_1$ contract in $\CP^2-C$. So $x=1$ and $\gamma_1=x^d=1$. 
Therefore $\pi_1(X-\cup D_i)=1$. The orbifold fundamental group $\pi_1^{\orb}(X)$ is the quotient with the conditions
$\gamma_i^{m_i}=1$. In any case, $\pi_1^{\orb}(X)=1$.
\end{proof}

By \cite[Corollary 10.2.11]{BG}, a Smale-Barden manifold which admits a K-contact structure necessarily satisfies the 
{\it G-K condition}, which means that, in terms of the expression (\ref{eqn:H2-1}):
 \begin{itemize}
 \item for every prime $p$, $t(p)=\#\{i\, | \, c(p^i)>0\}\leq k+1$,
 \item  $i(M)\in\{0,\infty\}$; if $i(M)=\infty$ ($M$ non-spin), then $t(2) \leq k$.
 \end{itemize}
In \cite[Question 10.2.1]{BG} it is asked whether a Smale-Barden manifold which 
satisfies the G-K conditon admits a Sasakian structure.
Write
 $$
 \mathbf{t}(M)=\max\{ t(p)|\,p \text{  prime}\} \leq k+1.
$$
The difficulty to obtain examples increases as we go to the upper bound, since we can always  discard surfaces from the
isotropy locus. The examples of \cite{CMRV,MRT} are instances where the upper
bound $\mathbf{t}(M)=k+1$ is achieved.

Note that the case $\mathbf{t} =0$ is that of torsion-free Smale-Barden manifolds, where
we only have regular Sasakian structures, and all G-K manifolds admit Sasakian structures.
The next case is $\mathbf{t}=1$, which is studied in detail in \cite{MT}. 
All G-K manifolds with $\mathbf{t}=1$ and $k\geq 1$ admit semi-regular Sasakian structures, and
hence the manifolds admitting Sasakian and K-contact structures are the same.
In the borderline case $\mathbf{t}=1,k=0$, the results in \cite{MT} are only partial
and touch upon open questions on symplectic $4$-manifold topology.

Write also
 $$
 \mathbf{c}(M)=\frac12 \max\{ c(p^i)\} =\max \{g_i\}.
 $$
in terms of the expression for $H_2(M,\ZZ)$ given in Theorem \ref{thm:H1=0}.
It is hard to get examples with low $\mathbf{c}(M)$ and $\mathbf{t}(M)=k+1$, $k=b_2(M)$.
In \cite{CMRV} there is given an example with $\mathbf{c}(M)=3$, and in \cite{M} there appears an
example with $\mathbf{c}(M)=2$.  Theorem \ref{thm:g=1} allows us to improve the bound.

\begin{corollary} \label{cor:g=1}
 There is a simply connected $5$-manifold $M$ admitting a (quasi-regular) Sasakian structure with 
 $\mathbf{t}(M)=b_2(M)+1$ and $\mathbf{c}(M)=1$. 
\end{corollary}

\begin{proof}
Let $b\geq 1$ and consider the orbifold $X$ constructed in Theorem \ref{thm:g=1}. It contains $b$ disjoint
complex curves of genus $1$ spanning the homology. Choose a prime $p$ and consider coefficients $m_i=p^i$.
Using Proposition \ref{prop:orb-constr}, we give $X$ an orbifold structure with isotropy locus given by the $D_i$ 
with coefficients $m_i$. We assign local invariants $b_i=1$, so that Theorem \ref{thm:constr-tool} 
allows to choose a Seifert bundle $M\to X$ whose Chern class is an orbifold K\"ahler form and $c_1(M/m)$ is primitive. Using
Theorem \ref{thm:H1=0} we have $H_1(M,\ZZ)=0$. Note that $\pi_1^{\orb}(X)=1$ by 
Proposition \ref{prop:pi1orb}, hence $\pi_1(X)=1$ and so $H_1(X,\ZZ)=0$.
By Theorem \ref{thm:H1=0} we have
 \begin{equation}\label{eqn:k=k}
  H_2(M,\ZZ)= \ZZ^k \oplus \big( \mathop{\oplus}\limits_{i=1}^{k+1} \ZZ_{m_i}^2 \big),
  \end{equation}
with $m_i=p^i$ and $k=b-1$. To see that $M$ is Smale-Barden, we need to compute $\pi_1(M)$.
There is an exact sequence $\ZZ \to \pi_1(M)\to \pi_1^{\orb}(X)=1$, hence $\pi_1(M)$ is abelian.
As $H_1(M,\ZZ)=0$, it must be $\pi_1(M)=0$. This concludes the result.
\end{proof}

\begin{corollary}\label{cor:extra}
 The manifold $M$ constructed in Corollary \ref{cor:g=1} is Sasakian quasi-regular but does not admit a Sasakian semi-regular structure.
 \end{corollary}
 
 \begin{proof}
 If $M$ admits a Sasakian semi-regular structure, then there is a Seifert fibration $M\to Y$, where
 $Y$ is a smooth K\"ahler surface. By Theorem \ref{thm:H1=0}, $b_1(Y)=0$, $b_2(Y)=k+1=b$, and the ramification locus 
 contains a collection of $b$ disjoint complex curves $D_i$  of genus $g_i=1$. By \cite[Theorem 29]{CMRV},
 this is impossible (put differently, the Conjecture \ref{conj} holds in the case of curves of genus $1$). 
 \end{proof}

We end up this section by extending the result of Theorem \ref{thm:g=1} to higher genus.

\begin{theorem} \label{thm:g>1}
Take $d\geq 3$ and $g=\frac{(d-1)(d-2)}{2}$. 
Then, for any $b\geq1$, there is a Kähler cyclic orbifold $X$ with $b_1=0$, $b_2=b$ and $b$ disjoint curves of genus $g$.
In addition, $X$ can be constructed in such a way that $\pi_1^\orb(X)=1$.
\end{theorem}

\begin{proof}
Consider inside the projective plane $\CP^2$, the family of curves, $C_\lambda$, $\lambda\in\CC$, given by:
 $$
 C_\lambda\equiv\left\{[x_0,x_1,x_2]\in\CP^2\, | \, F_\lambda(x_0,x_1,x_2)=x_0^{d-1}x_2-x_1^d+\lambda x_2^d=0\right\} .
 $$
These are all curves of degree $d$ and, therefore, they have genus $g=\frac{(d-1)(d-2)}{2}$. Now note that 
for $\lambda\neq 0$, $C_\lambda$ is smooth. Indeed, suppose $C_\lambda$ is not smooth at some point $[x_0,x_1,x_2]$. 
Then $0=\frac{\partial F_\lambda}{\partial x_1}=-dx_1^{d-1}$, from where $x_1=0$. Also $0=\frac{\partial F_\lambda}{\partial x_0}=(d-1)x_0^{d-2}x_2$
from where either $x_0=0$ or $x_2=0$. Finally, $0=\frac{\partial F_\lambda}{\partial x_2}=x_0^{d-1}+d\lambda x_2^{d-1}$.  Then
in any case, $x_0=0, x_2=0$.

Now let us see that the curves $C_\lambda$, $\lambda\neq 0$, intersect at a single point with multiplicity $d^2$.
Take $\lambda_1\neq\lambda_2$ and $p=[x_0,x_1,x_2] \in C_{\lambda_1}\cap C_{\lambda_2}$. Then
 \begin{align*}
 &x_0^{d-1}x_2-x_1^d+\lambda_1x_2^d=0,\\
 &x_0^{d-1}x_2-x_1^d+\lambda_2x_2^d=0.
 \end{align*}
From here, $\lambda_1x_2^d=\lambda_2x_2^d$, which implies $x_2=0$ and subsequently $x_1=0$. 
Therefore $p=[1,0,0]$, and hence
$C_{\lambda_1}\cdot C_{\lambda_2}=d^2p$.

Thus, selecting $b$ curves in the family $\{C_\lambda\}_{\lambda\neq 0}$, we get $b$ smooth curves $C_1,\ldots, C_b$ 
of genus $g$, each pair of them intersecting at the point $p_0=[1,0,0]$ with multiplicity $d^2$.
Now the same argument as that in the proof of Theorem \ref{thm:g=1} produces the 
Kähler cyclic orbifold $X$ with $b_1=0$, $b_2=b$ and $b$ disjoint curves $D_1,\ldots, D_b$ of genus $g$.

The proof of Proposition \ref{prop:pi1orb} works also in this case. We only need to choose the curves $C_\lambda$ close
to each other (with numbers $\lambda$ very close together), so that the homotopy classes on each $C_\lambda$ can
be pushed to $\bd U(C_{\lambda_0})$, for a fixed $\lambda_0$. Take a basis $\alpha_1,\beta_1\ldots, \alpha_g,\beta_g$
of $\pi_1(C_{\lambda_0})$. This basis can be given by loops well away from $p_0$. These can be contracted in $\CP^2-C_{\lambda_0}$,
because they can be defined by vanishing cycles in a suitable Lefschetz fibration in which $C_{\lambda_0}$ is a fiber.
\end{proof}

As in Corollary \ref{cor:g=1}, the orbifold in Theorem \ref{thm:g>1} serves to construct a 
simply connected $5$-manifold $M$ admitting a (quasi-regular) Sasakian
structure with 
 \begin{equation}\label{eqn:k=k2}
  H_2(M,\ZZ)= \ZZ^k \oplus \big( \mathop{\oplus}\limits_{i=1}^{k+1} \ZZ_{m_i}^{2g} \big),
  \end{equation}
with $m_i=p^i$ and $k=b-1\geq 0$, and $g=\frac{(d-1)(d-2)}2$ any triangular number. 
This manifold has 
 $\mathbf{t}(M)=b_2(M)+1$ and $\mathbf{c}(M)=g$.

\section{Quasi-regular vs semi-regular} \label{sec:4}
 
Let us give an easy example of a quasi-regular Sasakian Smale-Barden manifold that cannot be
semi-regular K-contact, which improves on Corollary \ref{cor:extra}.
   
\begin{proposition}\label{prop:quasi-vs-semi}
 There is a cyclic K\"ahler orbifold $\bar Y$ with $b_2=1$ and an embedded curve $D$ of genus $g=2$.
\end{proposition}

\begin{proof}
Let $Y=H_2$ be the Hirzebruch surface with invariant $n=2$, that
is $Y=\PP(\cO_{\CP^1} \oplus \cO_{\CP^1}(2))$, which is a $\CP^1$-bundle over $\CP^1$. The zero section 
$\s$ has $\s^2=2$. Let $f$ be the fiber. Therefore the section at infinity $\s_\infty\equiv \s-2f$
has $\s_\infty^2=-2$. The canonical class is   $K_Y=-2\s$ and the ample cone is generated by $\la \s , f\ra$. 

We take a curve $D\equiv 2\s+f$. This has
genus $2$ and $D^2=12$. It can be taken to be smooth. Let us give explicit equations.
In the affine part $X^o=X-\s_\infty$, which is the total space of the line bundle $\cO_{\CP^1}(2)$ over $\CP^1$,
we have a tautological coordinate $y$. Let $x$ be the affine coordinate of the basis. Then
we take the equation
 $$
  y^2x +a(x)y+b(x)=0,
  $$
where $a\in H^0(\CP^1, \cO(3))$, $b\in H^0(\CP^1, \cO(5))$. It intersects the section $\s_\infty$ at one
point $x=0$, $y=\infty$, and every fiber at two points. The condition to be smooth in the affine
part is that the discriminant $\Delta(x)=a(x)^2-4b(x)x$ has no double roots. At the point at infinity, we
can take the coordinate $1/y$, to see that the curve is smooth there.

Now we contract $\s_\infty$ and we get a cyclic orbifold $\bar Y$ with a point of order $2$. The 
projected curve $\bar D$ becomes
a genus $2$ curve going through the singular point.
\end{proof}

\begin{lemma}  \label{lem:piorb-ot}
The orbifold fundamental group of the manifold in Proposition \ref{prop:quasi-vs-semi} is trivial
if the isotropy coefficient of $\bar D$ is an integer $m$ with $\gcd(m,6)=1$.
\end{lemma}

\begin{proof}
Consider  the curves $D$ and $S=\s_\infty$  in $Y=H_2$. Let $\g,\d$ be the loops around $D$ and
$S$, and let $\alpha_1,\beta_1,\alpha_2,\beta_2$ the standard generators of $\pi_1(D)$.
Then $\gamma$ and $\delta$ commute since $S,D$ intersect transversally at one point. Now
in $\bd U(D)$, we have $\gamma^{12}=[\alpha_1,\beta_1]\, [\alpha_2,\beta_2]$,
using that $D^2=12$, and $\gamma$ is central. Also $\delta^2=1$ in $\bd U(S)$.
The loops $\alpha_i,\beta_i$ contract in $Y$, since it is simply-connected, hence it can be 
written in terms of $\g,\d$ in $\pi_1(Y-(D\cup S))$. Therefore $\pi_1(Y-(D\cup S))=\pi_1(\bar Y-\bar D)$ is
generated by commuting loops $\g,\d$ with $\g^{12}=1$, $\d^2=1$.

The orbifold fundamental group $\pi_1^{\orb}(\bar Y)$ is a quotient of such group, by imposing
the conditions $\d^2=\g$, $\g^{m}=1$, where $m$ is the isotropy coefficient of $\bar D$. 
This implies the result.
\end{proof}

\begin{theorem}\label{thm:quasi-vs-semi}
 There is a quasi-regular Sasakian Smale-Barden manifold, which is not semi-regular K-contact.
\end{theorem}

\begin{proof}
Apply Theorem \ref{thm:constr-tool}. Take an integer $m=m_1$ with $\gcd(m,6)=1$ and
local invariant $b_1=1$. Then use the Seifert bundle $M\to \bar Y$ with 
$c_1(M)=\frac{b_1}{m_1} [D_1]=\frac1m [D_1]$. This is an orbifold K\"ahler form, hence $M$ is 
Sasakian. The class $c_1(M/m)=[D_1]$ is primitive in $H^2(X-P,\ZZ)$, since it pairs
with the class $\s-2f \in H_2(X-P,\ZZ)$ giving $1$. Therefore we can apply
Theorem \ref{thm:H1=0} to prove that $H_1(M,\ZZ)=0$ and 
 $$
 H_2(M,\ZZ)=\ZZ^4_{m} \, .
 $$
Using Lemma \ref{lem:piorb-ot}, we prove that $\pi_1(M)=1$ and hence $M$ is a Smale-Barden 
manifold.
The manifold $M$ just constructed is quasi-regular Sasakian. It cannot be semi-regular
K-contact, since in \cite[Proposition 17]{MT} it is proved that this can only happen for 
$g=2$ being a triangular number. But this is not the case.
\end{proof}

\section{Rational symplectic $4$-manifolds}

The purpose of this section is to show that (at least some of) the manifolds constructed in Section \ref{sec:examples}
cannot be semi-regular K-contact.
This gives an example of quasi-regular Sasakian Smale-Barden $M$ that is not K-contact semi-regular
in the case $b_2(M)>0$ (that is, not rational homology spheres).
However, the proof is more technical than that of Theorem \ref{thm:quasi-vs-semi}, since
it uses Gromov-Witten and Seiberg-Witten theory for symplectic $4$-manifolds.

\begin{definition}[{\cite[Definition 2.2]{Li}}]
For a minimal symplectic $4$-manifold $(X,\omega)$, let $K$ be the symplectic canonical class. We 
define the Kodaira dimension as:
$$
 \kappa(X,\omega)=\left\{ \begin{array}{ll} -\infty \qquad & \text{if }K\cdot [\omega]<0 \text{ or } K^2<0, \\
 0& \text{if }K\cdot [\omega]=0 \text{ and } K^2=0, \\
 1& \text{if }K\cdot [\omega]>0 \text{ and } K^2=0, \\
 2& \text{if }K\cdot [\omega]>0 \text{ and } K^2>0.
 \end{array}\right.
 $$
 \end{definition}
 
 \begin{corollary}\label{cor:6}
 If $(X,\omega)$ is a symplectic $4$-manifold with $b_1=0$ and $K\cdot[\omega]<0$ then $X$ is rational (that is, 
 symplectomorphic to a rational algebraic surface).
 \end{corollary}
 
 \begin{proof}
 By \cite[Theorem 2.4]{Li}, if $(X,\omega)$ is minimal (that is, if it does not contain
an embedded symplectic sphere $S$ with $S^2=-1$, $K\cdot S=-1$), then $\kappa(X,\omega)=-\infty$
if and only if it is rational or ruled. If $b_1=0$ and $X$ is ruled, then it is an $S^2$-bundle over $S^2$,
that is a Hirzebruch surface, which is rational.

If $(X,\omega)$ is non-minimal, then there is a minimal symplectic manifold $(X',\omega')$ and a blow-up
map $\pi:X\to X'$. Suppose that $\pi$ is a single blow-up (the general case is done by repeating
the argument). Note that $b_1(X')=0$. Let $E$ be the exceptional divisor, and $K,K'$ the
canonical classes. Then $K=K'+E$ and $[\omega]=[\omega']-\alpha [E]$, for some $\alpha>0$. 
Then $K'\cdot  [\omega']=K\cdot [\omega]-\alpha < 0$. 
Hence $X'$ is rational, and so $X$ is also rational.
 \end{proof}

\begin{lemma} \label{lem:7}
Let $(X,\omega)$ be a compact symplectic $4$-manifold, let $D$ be a symplectic surface with $[D]^2>0$. 
Then there is a symplectic form $\omega'$ so that $[\omega']=[\omega]+ \lambda [D]$, for any $\lambda>0$.
\end{lemma}

\begin{proof}
By the symplectic tubular neighbourhood theorem, we can assume that a neighbourhood of $D$ modelled in a 
complex manifold, that is $U\subset L$, where $\pi:L\to D$ is a holomorphic line bundle of degree $m= [D]^2>0$.
We take the complex structure $J$ on $L$. The boundary of a unit circle bundle $S(L)$ in $L$ is a Sasakian manifold
with a contact form $\eta$ such that $d\eta=\pi^*(\omega_D)$. Let $r$ be the radial coordinate of $L$ (we have fixed
an hermitian metric). The form 
 $$
 \beta=d((1/2)r^2\eta)=r dr \wedge \eta+ (1/2)r^2 \omega_D
 $$
is a K\"ahler form for the cone, that is K\"ahler for $L$ except that it vanishes over the zero section.
Note that $r^2\eta$ is a well-defined form on $L$, so $\beta$ is exact.

Take a perturbation of $D$ as follows. We take a trivialization of $S(L)$ over all of $D$ but a point $p_0$,
we lift to construct $D'_1$ over $D-\{p_0\}$ at $r=\epsilon$, Then the boundary of $D'_1$ is the fiber $S(L_{p_0})$
with $m$ positive turns. We close it with $D'_2$ which is $m$ copies of $B_\epsilon(L_{p_0})$, and introduce the cycle 
$D'=D'_1\cup D'_2$. Then
 $$
 0=\int_{D'} \beta = \int_{D_1'} (1/2)\epsilon^2 \omega_D + \int_{B_\epsilon(0)} m r dr \wedge d\theta =
 (1/2)\epsilon^2 \mathrm{area}(D) - \epsilon^2 \pi m.
 $$
In particular, we note that $dr\wedge d\theta<0$ in the fiberwise direction. 

Now take a bump function $\rho(r)$ which is non-increasing, $\rho(r)\equiv 1$ for small $r$ and $\rho(r)\equiv 0$
for large $r$. Then take
 $$
 \Omega= d(\rho(r) \eta)= \rho'(r) dr \wedge \eta + \rho \, \omega_D\, .
 $$
As $\rho'(r)\leq 0$, this form is $J$-compatible semi-positive and compactly supported.
The first term is positive since $dr \wedge \eta<0$. Note that this is not exact because $\rho(r)\eta$ 
does not extend over $D$. Now take $\omega+ \lambda \Omega$, for $\lambda>0$, which solves the problem.
\end{proof}

\begin{theorem} \label{thm:8}
If $(X,\omega)$ is a compact symplectic $4$-manifold with $b_1(X)=0$, and $D$ is 
a symplectic surface with $[D]^2>0$ and $K\cdot D<0$, then $X$ is rational. 
\end{theorem}

\begin{proof}
By Lemma \ref{lem:7}, there is a symplectic form $\omega'$ with $[\omega']=[\omega]+ N [D]$, for $N>0$.
Now
 $$
  K\cdot [\omega']=K\cdot [\omega] + N \, K\cdot [D]<0
  $$
  for $N\gg 0$ large enough. Now by Corollary \ref{cor:6}, $(X,\omega')$ must be rational.
\end{proof}

\begin{corollary} \label{cor:9}
Suppose that $X$ is a symplectic $4$-manifold with $b_1(X)=0$, and with disjoint surfaces
$D_i$ of genus $1$ and spanning the homology $H_2(X,\QQ)$. Then $X$ is rational.
\end{corollary}

\begin{proof}
As $X$ is symplectic then $b^+\geq 1$. So there must be one of the surfaces, say $D_1$, with
$D_1^2>0$. As $D_1$ is a torus, then $K\cdot D_1=-D_1^2<0$. Apply now Theorem \ref{thm:8}.
\end{proof}

\begin{proposition} \label{prop:cor:10}
Suppose that $X$ is a symplectic $4$-manifold with $b_1=0$, and with disjoint surfaces
$D_i$ of genus $1$ and spanning the homology $H_2(X,\QQ)$. Then $b_2\neq 2$.
\end{proposition}

\begin{proof}
Suppose that $b_2=2$. 
By Corollary \ref{cor:9}, $X$ is rational, hence it must be a Hirzebruch surface $H_n$, $n\geq 0$.
Let $D_1,D_2$ be the disjoint symplectic genus $1$ curves. Then
 $$
 D_1\cdot D_2=0, \;\; D_1 \cdot (K+D_1)=0,
 $$
 imply by inspection of the intersection matrix that
 \begin{eqnarray}
 \label{eq:sum=0}
D_1+D_2+K = 0.
\end{eqnarray}

Let $\s,f$ be the basis of $H_2(X,\ZZ)$ given by the special section $\s$ and the fiber $f$ of the fibration, with
$$\s^2=-n,\s\cdot f=1, f^2=0.
$$
Then $K=-2\s -(n+2)f$. Write $D_1=a\s+b f$, with $a,b\in \ZZ$. Interchanging
$D_1$ and $D_2$ if necessary, we can assume $a\leq 1$ by \eqref{eq:sum=0}. 
The equation $D_1\cdot D_2=0$ 
simplifies to
 \begin{eqnarray}
 \label{eq:relation}
(a-1)(an-2b) = -2a.
 \end{eqnarray}
 Thus $a\neq 0,1$, so $a\leq -1$.
 Since \eqref{eq:relation} is an integer relation, we use $\gcd(a,a-1)=1$ to infer that $(a-1)\mid 2$,
 hence $a=-1$. Then \eqref{eq:relation} gives $n+2b=1$, and hence $b\leq 0$.
 Consider the Kähler class $[\omega]=x \s+ y f$ with $x>0, y>nx$ (since the surfaces $\s$ and $f$ 
 are symplectic, hence they pair positively with $[\omega]$). 
But then we find
\[
0<[\omega]\cdot D_1 = \underbrace{(nx-y)}_{<0}+bx.
\]
Since $b<0$ we obtain the required contradiction.
\end{proof}

\begin{remark}
Note that symplectic surfaces in a symplectic $4$-manifold can pair negatively.
Therefore, we cannot assume that $D_1\cdot f\geq 0$ in the proof above,
which would simplify the argument.

Also note that in the complex case, equation \eqref{eq:sum=0} leads to
a contradiction since it would provide an exact sequence
$0 \to \cO(K) \to \cO \to \cO_{D_1} \oplus \cO_{D_2} \to 0$. Using
that $H^0(K)=0$, $H^1(K)=0$, we would get that $H^0(\cO) = 
H^0( \cO_{D_1} )\oplus H^0(\cO_{D_2} )=\CC^2$, which is
not true. However, for almost complex $4$-manifolds there
is no analogue of this cohomology theory.
\end{remark}

Now we complete the proof that the manifold $M$ of Corollary \ref{cor:extra} 
is not semir-regular K-contact for $k=1$.

\begin{theorem}\label{thm:quasi-general}
 If a Smale-Barden manifold $M$ has $H_2(M,\ZZ)=\ZZ \oplus \big({\oplus}_{i=1}^{2}\, \ZZ_{p^i}^2\big)$, then it
 cannot be K-contact semi-regular.
 \end{theorem}

\begin{proof}
 Suppose that $M$ is K-contact semi-regular. Then there is a Seifert fibration
 $\pi:M\to Y$, where $Y$ is a smooth symplectic $4$-manifold with $b_1(Y)=0$, 
 $b_2(Y)=b=k+1=2$, with a collection of disjoint smooth symplectic 
 embedded surfaces $D_i$ of genus $1$, and spanning the homology $H_2(Y,\QQ)$.
 This is impossible by Proposition \ref{prop:cor:10}.
\end{proof}
 
 \begin{remark}
 Note that we have no chance to extend Theorem \ref{thm:quasi-general} to the Smale-Barden manifold 
 of (\ref{eqn:k=k2}) with $k=1$ and $g>1$. Our proof relies heavily on the rationality of symplectic
 $4$-manifolds, proved in Corollary \ref{cor:9}, which hinges on the fact that the surfaces are of genus $1$.
 \end{remark}

 \section{Null Sasakian structures}

\subsection{A description of   null Sasakian structures}

A smooth K3 surface is a simply connected complex surface with trivial canonical class $K_X$. The condition $K_X=0$ is equivalent to the existence of a nowhere vanishing holomorphic 2-form $\omega_X$. To define the orbifold version of this definition, we need orbifold homology $H_i^{\orb}(X)$ and the orbifold canonical class $K_X^{\orb}$ (see \cite[Chapter 4]{BG}).   

\begin{definition} 
An orbifold K3-surface is a cyclic orbifold $X$ such that $H_1^{\orb}(X)=0$ and $K_{X}^{\orb}=0$. 
\end{definition}

\begin{proposition}[{\cite[Proposition 10.2]{K}}] 
Let $\pi: M\rightarrow (X,\Delta=\sum_i(1-{1\over m_i})D_i)$ be a Seifert bundle with $M$ smooth. Assume that $H_1(M,\mathbb{Z})=0$ and that the orbifold $(X,\Delta)$ is a Calabi-Yau orbifold, that is $K_X+\Delta$ is numerically trivial. Then $K_X$ is trivial (and $\Delta=0$).
\end{proposition}

In addition, $M$ is simply connected if and only if $\pi_1^{\orb}(X)=1$
by \cite[Thm.\ 9.1]{K}.

\begin{proposition}[{\cite[Corollary 10.4]{K}}]\label{prop:minres} 
Let $\pi: M\rightarrow (X,\Delta)$ be a 5-dimensional Seifert bundle, $M$ smooth. Assume that $H_1(M,\mathbb{Z})=0$, and that $(X,\Delta)$ is a Calabi-Yau orbifold. Then
\begin{enumerate}
\item the minimal resolution of $X$ is a K3-surface,
\item $M$ is homeomorphic to the connected sum of at most 21 copies of $S^2\times S^3$.
\end{enumerate}
\end{proposition}

\begin{theorem}[{\cite[Theorem 10.3.8]{BG}}] 
If a Smale-Barden manifold $M$ admits a null Sasakian structure, then
\begin{enumerate}
\item any null Sasakian structure is quasi-regular, and, therefore, $M$ admits a structure of a Seifert bundle over an orbifold K3 surface,
\item $2\leq b_2(M)\leq 21$,
\item if $b_2(M)=21,$ the null structure is regular, that is, the Seifert bundle $M\rightarrow X$ is a smooth principal circle bundle over a K3 surface,
\item $M$ is spin,
\item $\pi_1^{\orb}(X)=1$.
\end{enumerate}
\end{theorem}

\begin{proof} 
It follows from Proposition \ref{prop:minres} and the discussion around it
 in a straightforward manner (see \cite[Section 10.3.2]{BG}).
\end{proof}

\begin{theorem}[{\cite[Corollary 10.3.11]{BG} and \cite[Theorem A]{CV}}] 
Any $M=\#_k(S^2\times S^3)$ with $2\leq b_2(M)\leq 21$ admits a null Sasakian structure, except, possibly, $b_2(M)=2$. 
\end{theorem}

¡\begin{proof} 
The proof for all cases except $k=17$  is based on the list of examples of orbifold K3 surfaces in \cite{R}. These are hypersurfaces in weighted projective spaces. In the case $k=17$ the author of \cite{CV} goes along the similar lines using the list of weighted complete intersections in \cite{IF}. The second Betti numbers are calculated in \cite[Example 10.3.10]{BG}  and in \cite{CV}, and it appears that all $k$ can be realized except   $k=2$.
\end{proof}

\subsection{Preparatory work to prove   Theorem \ref{thm:main-null}}

In order to construct a null Sasakian structure on a Smale-Barden manifold $M$ with $b_2(M)=2$,
we need a cyclic K3 orbifold $X$ with $b_2(X)=3$ and with the property $\pi_1^{\orb}(X)=1$. This follows from Theorem \ref{thm:H1=0}.
Thus, we begin with a construction of such an orbifold. 
We will use the method of lattices going back to Pyatetskii-Shapiro  and Shafarevich \cite{PSS} and developed in \cite{N} combined with methods of calculation of the fundamental groups of  smooth parts of K3 surfaces with cyclic singularities (that is, $\pi_1^{\orb}(X)$) from \cite{K-Z,S-Z}.  
We understand a  lattice as a free $\mathbb{Z}$-module endowed with a non-degenerate symmetric bilinear form with values in $\mathbb{Z}$. If $R$ is a lattice and $x,y\in R$, 
we write $x\cdot y$ for the value of this form on $(x,y)$. Assume that $\alpha\in R$ satisfies $\alpha^2=-2$. Then $\alpha$ determines an automorphism  $s_{\alpha}: R\rightarrow R$ by the formula
 $$
 s_{\alpha}(x)=x+(x\cdot\alpha)\alpha\, .
 $$

Assume that the bilinear form $(x,y)\rightarrow x\cdot y$ on $R$ is negative definite, and that $R$ is generated by elements of square $(-2)$.  Then the group generated by all $s_{\alpha}$ is the Weyl group, and free generators of square $(-2)$  constitute a root system, 
and one can associate with $R$ a Dynkin diagram in a standard way (see \cite{N} for a more detailed account). In particular, it is known that such lattices  correspond to the Dynkin diagrams which are disjoint unions of connected Dynkin diagrams of types $A_k,k\geq 1$, 
$D_l,l\geq 4$, $E_6,E_7,E_8$. Thus, we will call $R$ a root lattice.   

 Let $X$ be a smooth K3 surface. Then $H^2(X,\mathbb{Z})$ is an even unimodular lattice 
 with the intersection form of signature $(3,19)$ as the bilinear form in the definition. 
Consider complex elliptic surfaces $f: X\rightarrow\mathbb{C}P^1$ with a section $O$. The N\'eron-Severi lattice $NS(X)$ of $X$ is defined as 
$H^{1,1}(X)\cap H^2(X,\mathbb{Z})$. 
For elliptic surfaces over $\mathbb{C}P^1$, $NS(X)$ coincides with the Picard lattice $\operatorname{Pic}(X)$. 
The cohomology classes of $O$ and a generic fiber of $f$ generate a sublattice $U_f$ of rank $2$. Since $U_f$ turns to be a hyperbolic lattice, there is a decomposition $NS(X)= U_f \oplus W_f$
for some negative-definite even lattice $W_f$. 
 The Mordell-Weil group $MW_f$ of $X$ is understood as $NS(X)/T$, where $T$ is generated by $O$ and the fiber components \cite{Schutt-Shioda}.   
 
 In the sequel, we require that $X$ is extremal, that is, the Picard number $\rho(X)=20$ and the Mordell-Weil group $MW_f$ is finite. Such surfaces are classified in \cite{S-Z}. We will need a more detailed account of this classification. 
Let $R_f$ be the (finite) set of points $v\in\mathbb{C}P^1$ such that $f^{-1}(v)$ is reducible. For a point $v\in R_f$ the notation $f^{-1}(v)^{\#}$ means the union of reducible components of $f^{-1}(v)$ that are disjoint from $O$. 
It has been known since Kodaira's work that the cohomology classes of irreducible components of $f^{-1}(v)^{\#}$ form a negative definite root lattice $S_{f,v}$. The sum of disjoint components of the Dynkin diagrams corresponding to $S_{f,v}$ will be called the type $\tau(S_{f,v})$ of the lattice. Define the formal sum of the types of such lattices:
 $$
 \Sigma_f=\sum_{v\in R_f}\tau(S_{f,v}),
 $$
and denote by $S_f$ the direct sum of $S_{f,v}$.

Define 
 $$
  \operatorname{eu}(\Sigma_f)=\sum_{l\geq 1}a_l(l+1)+\sum_{m\geq 4}d_m(m+2)+\sum_{n=6}^8e_n(n+2),
  $$
where $a_l$, $b_m$ and $e_n$ denote the numbers of the connected components of the Dynkin diagrams of each type. Denote by $\Gamma_f$ the union of the zero section and all irreducible components  of $f^{-1}(v), v\in R_f$. For a point $v\in R_f$, we denote  the total fiber of $f$ over $v$ by
 $$
 \sum_{i=1}^{r_v}m_{v,i}C_{v,i}\, ,
 $$
where $m_{v,i}$ is the multiplicity of the irreducible component $C_{v,i}$ of $f^{-1}(v)$. It is known  
that $MW_f=W_f/S_f$ (cf.\ \cite[Prop.\ 6.42]{Schutt-Shioda} or \cite[Lem.\ 2.5]{S-Z}).

\begin{theorem}[{\cite[Claim 2, p.\ 36]{S-Z}}]\label{thm:shimada} 
Assume $MW_f=0$. Suppose that a configuration $\Gamma\subset\Gamma_f$ satisfies the following conditions:
\begin{enumerate}
\item[(Z1)] the number of $v\in R_f$ such that $C_{v,i}\subset\Gamma$ holds for any $C_{v,i}$ with $m_{v,i}=1$ is at most one,
\item[(Z2)] $\operatorname{eu}(\Sigma_f)\leq 23$.
\end{enumerate}
Then $\pi_1(X-\Gamma)=1$.
\end{theorem}

\subsection{Proof of Theorem \ref{thm:main-null}}

We begin with a construction of a K3 orbifold with $b_2(X)=3$ and $\pi_1^{\orb}(X)=1$.
In \cite{Shioda}, there is a Weierstrass form given for an elliptic K3 surface $X'$,
with a fibre of Kodaira type I$_{19}$, i.e.\ a cycle of 19 $(-2)$-curves.
Exactly one fibre component, say $\Theta_0$ connects with the zero section of the fibration, 
which provides another $(-2)$-curve $O$.
Omitting one of the two fibre components adjacent to $\Theta_0$,
we derive an $A_{19}$ configuration $\Gamma$ of $(-2)$-curves given by $O, \Theta_0, \Theta_1,\hdots,\Theta_{17}$.
This configuration is a part of the basis of $\operatorname{Pic}(X')$, 
since one can just add the remaining fiber component. Since $\operatorname{eu}(A_{19})=20$ 
and $MW_f=W_f/S_f=0$, we see that $X'$ satisfies the conditions $(Z1)$ and $(Z2)$ of Theorem \ref{thm:shimada}. Thus, $\pi_1(X'-\Gamma)=1$.

Contracting the 19 curves $O, \Theta_0$, $\Theta_1,...,\Theta_{17}$  successively, we obtain an  orbifold K3 surface $X$ with singularity $P$  of type 
$A_{19}$, $b_2(X)=3$ and 
 $$
 \pi_1(X'-\Gamma)=\pi_1(X-P)=\pi_1^{\orb}(X)=1.
 $$
 
Now we take a Seifert bundle $M\to X$ with primitive Chern class $c_1(M)\in H^2(X-P,\mathbb{Z})$.
This can be arranged by picking first a primitive class in $H^2(X-P,\mathbb{Z})$, 
then constructing the line bundle $M\rightarrow X-P$, and finally extending it as a Seifert bundle 
$\pi: M\rightarrow X$ across $P$. Moreover, we can assure that $c_1(M)$ is a Kähler class, so that
Theorem \ref{thm:seifert.} applies to produce a Sasakian structure on $M$.
As clearly $H_1(X,\ZZ)=0$, we have constructed a Seifert bundle satisfying the assumptions of 
Theorem \ref{thm:H1=0}. From these we get  $H_1(M,\ZZ)=0$,
and Theorem \ref{thm:main-null} follows using the fact that $\pi_1^{\orb}(X)=1$ implies $M$ to be simply connected,
so $M=\#_2(S^2\times S^3)$.
\qed

\begin{remark} 
In \cite{S-Z} the authors use the notation $\Delta$ instead of our $\Gamma$ (in this article $\Delta$ is a branch divisor).
\end{remark}

\begin{remark}
One can also work with other K3 surfaces $X'$,
such as the one from \cite[Remark 4.10]{S-Z}.
\end{remark}

\subsection{Remarks on the classification of null Sasakian structures}
The classification of null Sasakian structures amounts to a classification of orbifold K3 surfaces with  cyclic singularities satisfying the conditions $\pi_1^{\orb}(X)=1$. In general this may be out of reach, however, in case $\#_2(S^2\times S^3)$, one needs to classify cyclic orbifolds $X$ with $b_2(X)=3$ and trivial orbifold fundamental group. For example, it is concievable 
that one can classify extremal elliptic K3 fibrations with configurations of $(-2)$ curves of type $A$, using \cite{K-Z, S-Z,S-nodal}. 
For K3 surfaces of degree 2 and 6, this has been achieved in
\cite{Yang}, resp.\ \cite{Stegmann} (without considering fundamental groups).

\end{document}